\documentclass[10pt,a4paper]{amsart}

\usepackage{amsmath,amssymb}

 \input xy
 \xyoption{all}

\newtheorem{Theorem}{Theorem}[section]
\newtheorem{Lemma}[Theorem]{Lemma}
\newtheorem{Corollary}[Theorem]{Corollary}
\newtheorem{Proposition}[Theorem]{Proposition}
\newtheorem{Example}[Theorem]{Example}
\newtheorem{Remark}[Theorem]{Remark}
\newtheorem*{Definition}{Definition}

\def\ima{\operatorname{im}}
\def\ed{\operatorname{end}}
\def\bg{\operatorname{indeg}}
\def\dept{\operatorname{depth}}
\def\gen{\operatorname{gen}}
\def\coker{\operatorname{coker}}
 \def\reg{\operatorname{reg}}
 \def\deg{\operatorname{deg}}
 \def\adeg{\operatorname{adeg}}
\def\hdeg{\operatorname{hdeg}}
\def\indeg{\operatorname{indeg}}
\def\Ext{\operatorname{Ext}}
\def\ext{\operatorname{Ext}}
\def\tor{\operatorname{Tor}}
\def\Hom{\operatorname{Hom}}
\def\hom{\operatorname{Hom}}
\def\homgr{\operatorname{Homgr}}

\def\ol{\overline}

\def\mfrak{{\frak m}}

\def\im{{\mathfrak m}}

\def\d{{\delta}}
\def\b{{\beta}}

\def\ra{{\rightarrow}}

\newcommand{\bt}{\begin{Theorem}}
\newcommand{\et}{\end{Theorem}}
\newcommand{\bl}{\begin{Lemma}}
\newcommand{\el}{\end{Lemma}}
\newcommand{\bco}{\begin{Corollary}}
\newcommand{\eco}{\end{Corollary}}
\newcommand{\bp}{\begin{Proposition}}
\newcommand{\ep}{\end{Proposition}}
\newcommand{\bex}{\begin{Example}}
\newcommand{\eex}{\end{Example}}
\newcommand{\brm}{\begin{Remark}}
\newcommand{\erm}{\end{Remark}}

\begin{document}
\title{Castelnuovo-Mumford regularity of Ext modules and homological degree}
\author{Marc Chardin} 
\address{Institut de Math\'ematiques de Jussieu,
4, place Jussieu, F-75005 Paris, France.}
\email{chardin@math.jussieu.fr} 

\author{Dao Thanh Ha}
\address{Department of Mathematics,  University of Vinh, Vietnam}
\email{thahanh@yahoo.com}

\author{L\^e Tu\^an Hoa}
\address{Institute of Mathematics\\ 18 Hoang Quoc Viet Road\\ 10307 Hanoi, Vietnam}
\email{lthoa@math.ac.vn}
\thanks{The second and the third  authors were  supported  in part by the National Basic Research 
Program (Vietnam). The third author also would like to thank University of Paris 6 for the financial support and hospitality during his visit in 2007 when this work was started.}


\maketitle

\begin{abstract}  
Bounds for the Castelnuovo-Mumford regularity of Ext modules, over a polynomial ring over a field,
are given in terms of the initial degrees, Castelnuovo-Mumford regularities and number of generators of the two graded modules
involved. These general bounds are refined in the case the second module is the ring. Other estimates,
for instance on the size of graded pieces of these modules, are given.  We also derive
a bound on the homological degree in terms of the Castelnuovo-Mumford regularity. This answers positively
a question raised by Vasconcelos.
\end{abstract}

\section{Introduction}

Let $R$ be a polynomial ring in $n$ variables over a field and 
 $M$  be a finitely generated graded $R$-module. We are interested here in  estimating
 several invariants of $M$ in terms of the degrees in a presentation of $M$, or in terms
 of the Castelnuovo-Mumford regularity of $M$.  
  
 The homological degree was introduced by Vasconcelos and his students ten years ago (see \cite{DGV}). It is proved to be useful in many aspects (see, e.g.,  Chapter 9 in \cite{Va} and \cite{HHy}). One of our motivations was to answer positively a question of Vasconcelos \cite[page 261]{Va} on the existence of
 a polynomial bound on the homological degree of a module in terms of its regularity. In the case
 of a standard graded algebra $A$ of dimension $d>0$ with $n$ generators our bound is:
 $$  \hdeg (A)\leq {{\reg (A)+n}\choose{n}}^{2^{(d-1)^2}}. $$
  We derive this bound from an estimate of the homological degree in terms of the Hilbert 
 polynomial of a module, namely: assume $M$ has dimension $d>0$, regularity $r$ and Hilbert 
 polynomial $P$, then 
 $$
 \hdeg (M)\leq P(r)^{2^{(d-1)^2}}
 $$
 if $\dept (M)>0$ (the general case easily reduces
 to the result above).
 
 Another new result concerns estimates on the size of the coefficients of the Hilbert polynomial of a module $M$ in terms
 of its regularity and the degree of its quotient by $\dim M$ general linear forms. The bound in Theorem  \ref{Hilb3} refines 
 and extends earlier results of the third author, and is rather sharp.
 
 Several estimates are necessary to obtain these results. One concerns the regularity of the modules
 $\Ext^i_R(M,R)$ in terms of the regularity and the Hilbert polynomial of $M$. This problem was first studied in \cite{HaH}, and then continued in \cite{HHy} for the case $M$ is cyclic. It is an interesting problem because the regularity of the modules
 $\Ext^i_R(M,R)$ in some sense controls the behavior of the local cohomology module
$H^i_{\im}(M)$ in negative components.  The bound found in \cite[Theorem 14]{HHy} for the case $M$ being cyclic is a huge number and its proof required a rather complicated computation. Our bound here works for all modules and is much smaller, see 
 Theorem \ref{newregextMR}.  Its proof relies on the general estimates on $\reg (\Ext^i_R(M,N))$
 for a pair of modules proved in Section \ref{A} together with considerations on the effect of 
 the truncation of a module on its Ext's into $R$.  
An ingredient of the proof of \ref{newregextMR} that may be of use elsewhere is Corollary \ref{BettiHPoly}, which expresses the Betti numbers of a module with a
 linear resolution in terms of the values of its Hilbert polynomial at some integers around the regularity.
 
  Another type of estimates we establish concerns the vector space dimension of graded components of 
  Ext modules. Besides the rather obvious estimates mentioned in Section \ref{A}, we prove more delicate bounds
  in Theorem \ref{BoundExt1} in terms of the Hilbert polynomial of the module and of its regularity.  These can be used 
  in turn to estimate, via graded local duality, the size of the graded components of local cohomology modules
  as the third author first did in \cite[Theorem 3.4]{H} for ideals. Our proof here is completely different from that in \cite{H}: it is a direct proof, shorter and more elegant. 
  
The general estimates on the regularity of $\ext^i_R(M,N)$ for two graded modules $M$ and
$N$ are proved in Section \ref{A}. We use the fact that these modules are homology modules
of a complex of free modules whose shifts are controlled in terms of the ones appearing in
free $R$-resolutions of $M$ and $N$. The regularity of the homology of a
complex of free $R$-modules is estimated in terms of the regularities of cokernels of maps appearing
in the complex, which are in turn bounded by the general results of Fall,
Nagel and the first author in \cite{CFN}. In Section \ref{A2} we give a bound for  the regularity of the Ext modules.
 Section \ref{B} is devoted to the study of graded components of 
  Ext modules and the Hilbert coefficients. In the last Section \ref{C} we establish bounds  for the homological degree of a module (see Theorem \ref{B-hdeg1} and Theorem \ref{B-hdeg2}).

\section{General estimates on the regularity of Ext modules}\smallskip
\label{A}
Let $R$ be a polynomial ring in $n$ variables over a field, with $n\geq 2$, and 
 $M$ and $N$ be finitely generated graded $R$-modules.

Let $H_P$ denote the Hilbert function of a graded $R$-module $P$.

We will  estimate $\reg (\ext^i_R(M,N))$ in terms of the degrees appearing in free $R$-resolutions of $M$
 and $N$. For doing so, we first give a bound on the regularity of the homology modules of a graded complex
 of free $R$-modules in terms of the shift that appears in it. Convention: ${a\choose b} = 0$  if $a<b$.
 
\bl\label{regHpol}
Let $F^{\bullet}$ be a graded complex of free $R$-modules with $F^i:=\oplus _{f^i\leq j\leq b^i}R[-j]^{\b_{ij}}$. 
Set $T^i:=\sum_j \b_{i,j}$. Then, for any $i$,

(1) $\indeg (H^i(F^{\bullet}))\geq \indeg (F^{i})=f^{i}$,

(2)
$$
\reg (H^i(F^{\bullet}))\leq \max\{ b^{i},b^{i+1},[T^{i+1}(b^i-f^{i+1})]^{2^{n-2}}+f^{i+1}+2,[T^{i}(b^{i-1}-f^{i})]^{2^{n-2}}+f^{i}\} ,
$$

(3) for any $\mu \geq f^{i}$,
$$
\dim_k ((H^i(F^{\bullet}))_{\mu}\leq \dim_k (F^{i})_{\mu}\leq T^{i}{{\mu -f^i +n-1}\choose{n-1}},
$$

(4) for any $j$ and any $\mu\geq f^{i}+j$,
$$
\dim_k (\tor_j^R(H^i(F^{\bullet}),k))_{\mu}\leq T^{i}{{n}\choose{j}}{{\mu -f^i -j+n-1}\choose{n-1}}.
$$
\el
\begin{proof}
Statements (1) and (3) are obvious and (4) follows from (3) and from the fact that $k$ is resolved as an
$R$-module by the Koszul complex on the variables of $R$ (see the proof of \ref{dimTor}). 

We now prove statement (2).
Set $H^i:=H^i(F^{\bullet})$. The exact sequences $0\ra \ima (d^{i-1})\ra \ker (d^{i})\ra H^i \ra 0$,
$0\ra \ker (d^i)\ra F^i\ra F^{i+1}\ra \coker (d^i)\ra 0$ and
$0\ra  \ima (d^{i-1})\ra F^i\ra \coker  (d^{i-1})\ra 0$ imply the estimates
$$
\begin{array}{rl}
\reg (H_i)&\leq \max\{ \reg ( \ker (d_{i})),\reg (\ima (d^{i-1}))-1\} \\
	&\leq \max\{ \reg (F^i),\reg (F^{i-1})+1,\reg (\coker (d^i))+2,\reg (\coker (d^{i-1}))\} .\\
\end{array}
$$
By \cite[Theorem 3.5]{CFN}, the presentation $F^{j}\ra F^{j+1}\ra \coker (d^{j})\ra 0$ shows that, for any $j$,
$$
\reg (\coker (d^{j}))\leq [T^{j+1}(b^j-f^{j+1})]^{2^{n-2}}+f^{j+1}.
$$
The inequality in (2) follows.
\end{proof}

\bco\label{HLC}
Let $F^{\bullet}$ be a graded complex of free $R$-modules with $F^i:=R[r+i]^{T_{i}}$. Then,

(1) $\indeg (H^i(F^{\bullet}))\geq \indeg (F^{i})=-r-i$,

(2)
$$
\reg (H^i(F^{\bullet}))\leq \max\{ T_{i+1}^{2^{n-2}}+1, T_{i}^{2^{n-2}}\} -r-i,
$$

(3) for any $\mu \geq -r-i$,
$$
\dim_k ((H^i(F^{\bullet}))_{\mu}\leq \dim_k (F^{i})_{\mu}\leq T_{i}{{\mu +r+i +n-1}\choose{n-1}},
$$

(4) for any $j$ and any $\mu\geq -r-i+j$,
$$
\dim_k (\tor_j^R(H^i(F^{\bullet}),k))_{\mu}\leq T_{i}{{n}\choose{j}}{{\mu +r+i -j+n-1}\choose{n-1}}.
$$
\eco

For a finitely generated graded $R$-module $P$, set $T_i^P:=\dim_k \tor_i^R(P,k)$, $f_i^P:=\indeg ( \tor_i^R(P,k))$ and $b_i^P:=\reg ( \tor_i^R(P,k))$ (recall that $\indeg (0)=+\infty$ and $\reg(0)=-\infty$). 

\bt\label{regExtMN} Let $M$ and $N$ be finitely generated graded modules over the polynomial ring $R$.
With notations as above, 
set $T^i=\sum_{p-q=i}T_p^MT_q^N$, $r_M:=\reg (M)-\indeg (M)$, $r_N :=\reg (N)-\indeg (N)$ and
$\d :=\indeg (M)-\indeg (N)$.  Then, for any $i$,

(1) $\indeg (\ext^i_R(M,N))\geq e_i:=\indeg (N)-\reg (M)-i$, and equality holds for some $i$,

(2)
$$
\reg (\ext^i_R(M,N))+i\leq  (r_M+r_N+1)^{2^{n-2}} \max\{ T^{i} ,T^{i+1}\}^{2^{n-2}} +1-\d ,
$$

(3) for any $\mu \geq e_i$,
$$
\dim_k (\ext^i_R(M,N))_{\mu}\leq T^{i}{{\mu -e_i+n-1}\choose{n-1}},
$$

(4) for any $j$ and any $\mu \geq e_i+j$, 
$$
\dim_k (\tor_j^R(\ext^i_R(M,N),k))_{\mu}\leq T^{i}{{n}\choose{j}}{{\mu -e_i -j+n-1}\choose{n-1}}.
$$
\et

\begin{proof} For (1), see \cite[3.3]{CD} 
We now prove (2), for which we may, and will,  assume that $\indeg (M)=\indeg (N)=0$. Let $F^M_\bullet$ (resp. $F^N_\bullet$) be a minimal free $R$-resolution of M (resp. $N$) and set 
$C^\bullet :=\homgr_R(F^M_\bullet,F^N_\bullet )$.
Then $\ext^i_R(M,N)\simeq H^i(C^\bullet )$. One has,
$$\begin{array}{ll} f^i  := & \indeg (C^i)=\min_{p-q=i}\{ f_q^N -b_p^M\}\geq -i-\reg (M), \\
b^i:= & \reg (C^i)=\max_{p-q=i}\{ b_q^N-f_p^M\}\leq -i+\reg (N).
\end{array}$$
Set $K:=\reg (M)+\reg (N)+1$ and $\epsilon^{i}:=f^{i}+i+\reg (M)\geq 0$. By Corollary \ref{HLC}, it follows that
$$ \begin{array}{rl}
\reg (\ext^i_R(M,N))&\leq \max \{ b^i,\ b^{i+1},\ [T^{i+1}(b^i-f^{i+1})]^{2^{n-2}}+f^{i+1}+2,\\ & \hskip1cm [T^{i}(b^{i-1}-f^{i})]^{2^{n-2}}+f^{i}\} \\
&\leq \max \{ \reg (N)-i,\ [T^{i+1}(K-\epsilon^{i+1} )]^{2^{n-2}}-i+1+\epsilon^{i+1},\\
& \hskip1cm [T^{i}(K-\epsilon^{i} )]^{2^{n-2}}-i+\epsilon^{i}\} \\
&\leq \max\{ KT^{i} ,\ KT^{i+1}\}^{2^{n-2}}+1-i.
\end{array}
$$

Finally (3) and (4) follow from the estimates in Lemma \ref{regHpol} (3) and (4).
\end{proof}

Let $\mu (P)$ denote the minimal number of generators of a module $P$. The following lemma, in the spirit of some results above, can be used together with estimates on the regularities of $M$ and $N$ (see \cite{CFN}) to bound the regularity of $\ext^i_R(M,N)$ in terms of presentations of $M$ and $N$.

\bl\label{dimTor}
For any $i$,
$$
\dim_k \tor_i^R(M,k)\leq   \mu (M){{n}\choose{i}}{{\reg (M)-\indeg (M)+n}\choose{n}}.
$$
\el

\begin{proof}   We may assume that  $\indeg (M)\geq 0$. Then one has $\tor_{i}^{S}(M,k)\simeq H_i(x;M)$.
In particular $\dim_k (\tor_{i}^{S}(M,k))_\mu \leq {{n}\choose{i}}\dim_k (M_{\mu -i})\leq {{n}\choose{i}}  \mu (M) {{\mu -i+n-1}\choose{n-1}}$. It follows that 
$$
\dim_k \tor_i^R(M,k)\leq {{n}\choose{i}}  \mu (M)\sum_{\mu =i}^{\reg (M)+i} {{\mu -i+n-1}\choose{n-1}},
$$
from which the claimed inequality follows.
\end{proof}

\section{Refined estimate for the regularity of the modules $\ext^i_R(M,R)$}\smallskip
\label{A2}

Recall that $R$ is a standard graded polynomial ring in $n$ variables over a field $k$,  and let $\im$ be the maximal graded ideal of $R$. Let  $M$ be a finitely generated graded $R$-module. Set $\ol{M} :=M/H^0_\im (M)$ and $\Gamma M:=D_\im (M)\simeq
\oplus_{\mu} H^0({\bf P}^{n-1},\tilde{M}(\mu ))$.

We first make some remarks on the truncation of $M$. 

Let $t$ be an integer and set $M':=M_{\geq t}$. For example, if $M=R/I$ and $t\geq 0$, then 
$M':=\im^t/(I\cap\im^t)$. 

One has $\ext^i_R(M',R)=\ext^i_R(M,R)$ for $i<n-1$, and $\Gamma M=\Gamma M'$. We also have $H^0_\im (M')=H^0_\im (M)_{\geq t}$ and the commutative diagram
$$
\xymatrix{
0\ar[r]&H^0_\im (M')\ar[r]\ar[d]&M'\ar[r]\ar[d]&\Gamma M'\ar[r]\ar^{\simeq}[d]&H^1_\im (M')\ar[r]\ar[d]&0\\
0\ar[r]&H^0_\im (M)\ar[r]&M\ar[r]&\Gamma M\ar[r]&H^1_\im (M)\ar[r]&0\\
}
$$
shows that $H^1_\im (M')_{\geq t}=H^1_\im (M)_{\geq t}$ and $H^1_\im (M')_{\mu}=(\Gamma M)_{\mu}$ for $\mu <t$.

It follows that $\reg (M')=\max\{ t,\reg (M)\}$.  For a graded $R$-module $N$, set $N_{<t}:=N/(N_{\geq t})$. One has an exact sequence:
$$
0\ra H^0_\im (M)_{<t}\ra M_{<t}\ra H^1_\im (M')\ra H^1_\im (M)\ra 0,
$$
which gives by duality an exact sequence
$$
0\ra \ext^{n-1}_R(M,R[-n])\ra \ext^{n-1}_R(M',R[-n])\ra ^*\hom_R(M/H^0_\im (M),k)_{>-t}\ra 0,
$$
that in turn induces the exact sequence
\begin{equation*}
\xymatrix{
0\ar[r]& H^0_\im (\ext^{n-1}_R(M,R))\ar[r]& H^0_\im (\ext^{n-1}_R(M',R))\ar[r]& ^*\hom_R(M/H^0_\im (M),k)[n]_{>-t-n}&\\}
\end{equation*}
\begin{equation}\label{Eregext}
\xymatrix{
&\ar[r]& H^1_\im (\ext^{n-1}_R(M,R))\ar[r]& H^1_\im (\ext^{n-1}_R(M',R)\ar[r]&0.\\}
\end{equation}

Assume $t>\indeg (M)$ (i.e. $M'\not= M$). Then  $^*\hom_R(M/H^0_\im (M),k)[n]_{>-t-n}$ is of finite length supported in degrees $\in [-t-n+1, -\indeg (M/H^0_\im (M))-n]$. 
 It shows that
$$
\reg ( \ext^{n-1}_R(M,R))\leq \max\{ \reg ( \ext^{n-1}_R(M',R)),-\indeg (M/H^0_\im (M))-n\} .
$$

We gather direct consequences of the above facts in the following

\brm\label{regexttrunc}
Let $M':=M_{\geq t}$, then 

(i) $\reg (M')=\max\{ t,\reg (M)\}$,

(ii) $\ext^i_R(M',R)=\ext^i_R(M,R)$ for $i<n-1$,

(iii) $\reg ( \ext^{n-1}_R(M,R))\leq \max\{ \reg ( \ext^{n-1}_R(M',R)),-\indeg (M)-n\}$,

(iv) $\ext^n_R(M',R)=\ext^n_R(M,R)_{\leq -n-t}$,

(v)  $\ext^n_R(M,R)$ is a module of  finite length whose intitial degree is $-\ed (H^0_\im (M))-n\geq -\reg (M)-n$ and whose regularity is  $\bg (H^0_\im (M))-n \leq -\indeg (M)-n$.
\erm

Furthermore, 
$$
\mu (M')=\dim_k (\tor_0^R(M',k))=\dim_k (\tor_0^R(M',k))_r=H_M(r)\leq \mu (M){{r+n-1}\choose{n-1}}.
$$

\brm\label{noethnorm}
Let $M$ be a finitely generated graded $R$-module of dimension $d$. To estimate regularity,
we may assume $k$ is infinite. In this case, let $S$ be a polynomial ring in $d$ variables over
$k$ inside 
$R$  such that $M$ is finite over $S$. We may assume that $S=k[X_1,\ldots ,X_d]$. One has 
a graded isomorphism of $S$-modules 
$$
\ext^i_R(M,R)\simeq \ext^{i-n+d}_S (M,S)[n-d].
$$ 
It follows that $\ext^i_R(M,R)=0$ for $i<n-d$ and 
$$
\reg (\ext^i_R(M,R))=\reg (\ext^{i-n+d}_S (M,S))-(n-d),\quad \forall i.
$$
\erm

Notice that this last equality can be written as
$$
\reg (\ext^i_R(M,R))+i=\reg (\ext^{i-n+d}_S (M,S))+(i-n+d),\quad \forall i.
$$

We will need the following formula for the Betti numbers of a module with a linear resolution, which may be of use for other applications. 

\bp\label{tortrunc}
Let $M$ be a finitely generated graded $R$-module 
with $\reg (M)=\indeg (M)=:r$ and $H^0_\im (M)=0$. 
Set $M':=(M/lM)_{\geq r+1}$, for a linear non zero-divisor $l$.
Then, 
 $$
 \dim_k \tor_i^R(M,k)={{n-1}\choose{i}}P_M(r)-\dim_k \tor_{i-1}^{R/lR}(M',k).
 $$
\ep

\begin{proof} Recall that if $N$ is a module with $\indeg (N)=\reg (N)=s$, then
$\tor_i^R (N,k)$ is concentrated in degree $s+i$ for all $i$.
Notice that $\reg (M_{\geq r+1})=\indeg ( M_{\geq r+1})=r+1$, that $M'=0$ if $\dim M=1$, and
that $\reg (M')=\indeg (M')=r+1$ if $\dim M\geq 2$. We induct on $i$. The case $i=-1$ 
is trivially satisfied ($i=0$ is also clear). The exact sequences
$$ 0\ra  M(-1)\ra M_{\geq r+1}\ra M' \ra 0 $$
and
$$ 0\ra  M_{\geq r+1}\ra M \ra M_{r}\ra 0 $$
induce exact sequences
$$ 0\ra  \tor_i^R(M,k)\ra  \tor_i^R(M_{\geq r+1},k)\ra  \tor_i^R(M',k)\ra 0  $$
 and
 $$ 0\ra  \tor_{i+1}^R(M,k)\ra  \tor_{i+1}^R(M_r,k)\ra  \tor_{i}^R(M_{\geq r+1},k)\ra 0, $$
 which shows that
 $$  \begin{array}{rl}
  \dim_k \tor_{i+1}^R(M,k)&= \dim_k \tor_{i+1}^R(M_r,k)- \dim_k \tor_{i}^R(M,k)- \dim_k \tor_{i}^R(M',k)\\
  &= {{n}\choose{i+1}}P(r)- \dim_k \tor_{i}^R(M,k)-\dim_k \tor_{i}^{R/lR}(M',k)\\
& \hskip1cm -\dim_k \tor_{i-1}^{R/lR}(M',k)\\
   &= {{n}\choose{i+1}}P(r)-{{n-1}\choose{i}}P_M(r)-\dim_k \tor_{i}^{R/lR}(M',k)\\
    &= {{n-1}\choose{i+1}}P_M(r)-\dim_k \tor_{i}^{R/lR}(M',k)\\
\end{array}
  $$
  by induction on $i$.
\end{proof}

For a polynomial $P$, 
set $\Delta P(t):=P(t)-P(t-1)$ and $\Delta^i P:=\Delta (\Delta^{i-1} P)$.

\bco\label{BettiHPoly}
Let $M$ be a finitely generated graded $R$-module of dimension $d$
with $\reg (M)=\indeg (M)=:r$ and $H^0_\im (M)=0$. 
Then, 
 $$
 \dim_k \tor_i^R(M,k)=\sum_{\ell =0}^{\min\{ i,d-1\}}(-1)^{\ell}{{n-\ell -1}\choose{i-\ell}}\Delta^{\ell}P_M(r+\ell ).
 $$
\eco

\begin{proof} 
Notice that $M'$ has positive depth, because $M/lM$ has regularity $r$. As $M'$
has regularity $r+1$, Hilbert polynomial $\Delta^1 P_M$ and $R/lR$ is isomorphic to 
a polynomial ring in $n-1$ variables, the claim  follows by induction.
\end{proof}

\bt\label{newregextMR}
Set $d:=\dim M$ and $\bar{r} :=\reg (M/H^0_\im (M))$. 

(1) If $d<2$, then $\reg (\ext^i_R(M,R)\leq -\indeg (M)-i$ for any $i$.

(2) If $d\geq 2$, then

$\quad$ (a) $\reg (\ext^n_R(M,R)+n\leq -\indeg (M)$,

$\quad$ (b) $\reg (\ext^{n-1}_R(M,R))+(n-1)\leq \max\{ P_M (\bar{r})-\Delta^1P_M (\bar{r})-\bar{r}, - \indeg(M) -1\}$,

$\quad$ (c) for $i>1$
$$ \reg (\ext^{n-i}_R(M,R))+(n-i) \leq \left[  C_{d,d-i}\; P_M (\bar{r}) \right] ^{2^{d-2}}-\bar{r}+1,$$
with $C_{d,j}:=\max\{ {{d-1}\choose{j}},{{d-1}\choose{j+1}} \}$.
\et

\begin{proof}   (1) Recall that $\ext^i_R(M,R)=0$ for $i<n-d$. By Remark \ref{regexttrunc} (v), it  remains  to
check that the inequality holds when $d=1$ and $i=n-1$. When $d=1$, $\ext^{n-1}_R(M,R)\simeq \ext^{n-1}_R(M/H^0_\im (M),R)$, which shows  that 
$\reg (\ext^{n-1}_R(M,R))=-\indeg (M/H^0_\im (M))-(n-1)\leq -\indeg (M)-(n-1)$, because $M/H^0_\im (M)$ is Cohen-Macaulay of 
dimension $1$.

(2)(a) was proved in Remark \ref{regexttrunc} (v). For (2)(b) and (2)(c), by Remark \ref{regexttrunc} (i)-(iii) and Remark \ref{noethnorm}
we are reduced to show this estimate for $M$ with $\indeg (M)=\reg (M)$ and $d=n$. Applying Corollary \ref{HLC}
to the $R$-dual of a minimal free $R$-resolution of $M$, we deduce that, setting $T_i:=\dim_k \tor_i^R(M,k)$, one
has 
$$
\reg (\ext^i_R(M,R))\leq \max\{ T_i^{2^{n-2}},T_{i+1}^{2^{n-2}}+1\} -\bar{r}-i.
$$
Hence the conclusion follows from Proposition \ref{tortrunc}.
\end{proof}

\section{Hilbert function and Hilbert coefficients}\smallskip
\label{B}

In this section we will estimate graded components of the Hilbert function of $\Ext^i_R(M,R)$. Based on such an estimation we will give bounds for the Hilbert coefficients in terms of the Castelnuovo-Mumford regularity of $M$.

\bl \label{BoundExt3} Let $\ol{M}:=M/H^0_\im (M)$ and $\ol{r} = \reg (\ol{M})$. Then

\begin{itemize}
\item[(i)] $P_M(t)  = H_{\ol{M}}(t)$ for all $t\geq \ol{r}$ and $P_M(t)$ is  increasing for all $t\geq \ol{r} - 1$. 
\item[(ii)] If $\dim M \geq 1$, then $H_{\ol{M}}(\ol{r}) \geq \deg(M)$.
\end{itemize}
\el

\begin{proof}
 (i) By the  Grothendieck-Serre formula,
\begin{equation}\label{EGS} H_{\ol{M}}(t) - P_{\ol{M}}(t) = \sum_{i=1}^d (-1)^i \ell (H^i_{\mfrak}({\ol{M}})_t).
\end{equation}
 This implies that $P_M(t) = P_{\ol{M}}(t) = H_{\ol{M}}(t)$ for all $t\geq \ol{r}$. Since $H_{\ol{M}}(t)$ is an increasing function, $P_{\ol{M}}(t)$ is also increasing for all $t\geq \ol{r}$. If $d\leq 1$, then $P_M(t)$ is a constant. Let $d\geq 2$ and $l$ be a generic linear form. Then
 $$P_M(\ol{r}) - P_M(\ol{r}-1) = P_{M/lM}(\ol{r}) = H_{\ol{M/lM}}(\ol{r}) \geq 0.$$

(ii) If $\dim M =1$, then by (i) $H_{\ol{M}}(\ol{r}) = P_M(\ol{M}) = \deg(M)$. If $d = \dim M \geq 1$, let $l_1,...,l_{d-1}$ be a generic linear forms. Since
$\reg (\ol{M}/(l_1,...,l_{d-1})\ol{M}) \leq \ol{r}$, the above remark implies that
$$
H_{\ol{M}}(\ol{r}) \geq H_{\ol{ \ol{M}/(l_1,...,l_{d-1})\ol{M} }}( \ol{r}) = P_{\ol{ \ol{M}/(l_1,...,l_{d-1})\ol{M} }}( \ol{r}) = \deg (M).
$$
\end{proof}

\bt \label{BoundExt1}
Let $M$ be a finitely generated graded $R$-module of dimension $d\geq 1$. Let $l_1,\ldots ,l_d$ be a filter regular sequence of linear forms on $M$ and $M_j:=M/(l_1,\ldots l_j)M$. Set $\ol{M_j}:=M_j/H^0_\im (M_j)$ and $\bar{r}_j:=\reg (\ol{M_j})$. Then for $i>0$, $\indeg (\ext^{n-i}_R(M,R))\geq -\bar{r}_{i-1}-n+1$ and
$$
\dim_k  \ext^{n-i}_R(M,R)_{\mu}\leq {{\mu +\bar{r}_{i-1}+n-1}\choose{i-1}}\Delta^{i-1}P_{M}(\bar{r}_{i}-1).
$$
\et

\begin{proof}  First notice that $\ext^i_R(M_j,R)=\ext^i_R(\ol{M_j},R)$ for $i\not= n$.

Set $N_{j+1}:=\ol{M_j}/l_{j+1}\ol{M_j}$. One has $\reg (N_{j+1})=\reg (\ol{M_j})=\bar{r}_j$ and
$$N_{j+1} \cong M_j/(l_{j+1}M_j + H^0_\im (M_j) ) \cong M_{j+1}/((l_{j+1}M_j + H^0_\im (M_j) )/ l_{j+1}M_j ).$$
Noticing that the module $U:= (l_{j+1}M_j + H^0_\im (M_j) )/ l_{j+1}M_j \cong H^0_\im (M_j)/(H^0_\im (M_j)\cap l_{j+1}M_j )$ is of finite length, we get that $U$ is a submodule of $H^0_\im (M_{j+1})$ and $H^0_\im (M_{j+1})/U \cong H^0_\im (N_{j+1})$. Hence
\begin{equation} \label{BExt1a}
\ol{M_{j+1}} \cong (M_{j+1}/U)/(H^0_\im (M_{j+1})/U) \cong N_{j+1}/H^0_\im (N_{j+1}),
\end{equation}
 which also shows that $\bar{r}_{j+1}\leq \bar{r}_{j}$.

Now we show  by induction on $i\geq 1$ that
$$\indeg(\ext^{n-i}_R(M_j,R)) \geq \bar{r}_{j+i-1}-n + i$$
and  that
$$ \dim_k  \ext^{n-i}_R(M_j,R)_{\mu} \leq {{\mu +\bar{r}_{j+i-1}+n-1}\choose{i-1}}P_{M_{j+i-1}}(\bar{r}_{j+i}-1)$$
for all $j\geq 0$.

Let $i=1$. In this case, by Lemma \ref{BoundExt3}(i),  $H_{\ol{M_j}}(\nu )=P_{\ol{M_j}}(\nu )=P_{M_j}(\nu )$ for $\nu \geq \reg (\ol{M_j})$. Recall that $P_{M_{j+1}}(\nu )=P_{M_j}(\nu )-P_{M_j}(\nu -1)$ for any $\nu$. 
The exact sequence 
$$ 0\ra  \ol{M_j} (-1)\ra \ol{M_j} \ra N_{j+1}\ra 0 $$
 induces, for $i<n$, an exact sequence 
\begin{equation} \label{BExt1a0}
\cdots \ra \ext^i_R(M_j,R)\ra  \ext^i_R(M_j,R)(1)\ra  \ext^{i+1}_R(N_{j+1},R), 
\end{equation}
which shows that, for $i<n-1$, 
\begin{equation} \label{BExt1aa}
\dim_k  \ext^i_R(M_j,R)_{\mu}\leq \sum_{\nu <\mu} \dim_k  \ext^{i+1}_R(N_{j+1},R)_{\nu}.
\end{equation}
and
\begin{eqnarray}
\dim_k  \ext^{n-1}_R(M_j,R)_{\mu}&\leq & \sum_{\nu <\mu} \dim_k  \ext^{n}_R(N_{j+1},R)_{\nu} \nonumber \\
&= &\sum_{\nu <\mu} \dim_k H^0_\im (N_{j+1})_{-\nu -n}\nonumber\\
&= & \sum_{\nu =-n-\mu+1}^{\ed (H^{0}_{\im}(N_{j+1}))}\dim_k  H^0_\im (N_{j+1})_{\nu}\nonumber\\
&\leq &  \sum_{\nu \leq \bar{r}_j} (H_{N_{j+1}}(\nu )-H_{\ol{M_{j+1}}}(\nu ))  \nonumber \\
&=& H_{\ol{M_j}}(\bar{r}_{j})-\sum_{\nu \leq \bar{r}_{j}} H_{\ol{M_{j+1}}}(\nu ) \nonumber \\
&\leq &  P_{M_j}(\bar{r}_{j})-\sum_{\bar{r}_{j+1}\leq \nu \leq \bar{r}_{j}} P_{M_{j+1}}(\nu ) \ {\rm (by \ Lemma \ \ref{BoundExt3}(i))} \label{BExt1b} \\
&=& P_{M_j}(\bar{r}_{j})-\sum_{\bar{r}_{j+1}\leq \nu \leq \bar{r}_{j}} (P_{M_{j}}(\nu )-P_{M_{j}}(\nu -1)) \nonumber\\
&=& P_{M_j}(\bar{r}_{j+1}-1). \label{BExt1c}
\end{eqnarray}
 By (1) in Theorem \ref{regExtMN},  $\indeg(\ext^{n-1}_R(M_j,R)) \geq  - \bar{r}_j - n +1$. This means $\ext^{n-1}_R(M_j,R)_\mu =0$ for all $\mu \leq - \bar{r}_j - n$. For $\mu \geq  - \bar{r}_j - n +1$, ${\mu +\bar{r}_j+n-1\choose 0}= 1$. Hence (\ref{BExt1c}) implies  the claim for $i=1$ and all $j$. 

Let $i\geq 2$. Notice that 
 $$
 \ext^{n-i+1}_R(N_{j+1},R)\simeq \ext^{n-i+1}_R(\ol{M_{j+1}},R)\simeq \ext^{n-i+1}_R(M_{j+1},R).
 $$
 Furthermore, $M_{j+1}$ is of dimension $d-j-1$ and $l_{j+2},\ldots ,l_d$ is a filter regular
 sequence on $M_{j+1}$. 
 One has
\begin{eqnarray}
\dim_k  \ext^{n-i}_R(M_j,R)_{\mu}&\leq & \sum_{\nu <\mu} \dim_k  \ext^{n-i+1}_R(N_{j+1},R)_{\nu}  \ \ {\rm (by\ (\ref{BExt1aa}) \ and \ (\ref{BExt1a}))} \nonumber \\
&=& \sum_{\nu <\mu} \dim_k  \ext^{n-(i-1)}_R(M_{j+1},R)_{\nu}\nonumber 
\end{eqnarray}
By induction hypothesis, $\indeg (\ext^{n-(i-1)}_R(M_{j+1},R))\geq -\bar{r}_{(j+1)+(i-2)}-n+i-1= - \bar{r}_{j+i-1}-n+i-1$ and
therefore
\begin{eqnarray}
\dim_k  \ext^{n-i}_R(M_j,R)_{\mu}&\leq &   \sum_{\nu =-\bar{r}_{j+i-1}-n + i-1}^{\mu -1}\dim_k  \ext^{n-(i-1)}_R(M_{j+1},R)_{\nu}\nonumber \\
&\leq& \sum_{\nu =-\bar{r}_{j+i-1}-n +i-1}^{\mu -1}
{{\nu +\bar{r}_{(j+1)+(i-2)}+n-1}\choose{i-2}}\Delta^{i-2}P_{M_{j+1}}(\bar{r}_{(j+1)+(i-1)}-1) \nonumber \\
&=& \sum_{\nu =i-2}^{\mu +\bar{r}_{j+i-1}+n-2}{{\nu }\choose{i-2}}P_{M_{j+i-1}}(\bar{r}_{j+i}-1)\nonumber \\
&=&{{\mu +\bar{r}_{j+i-1}+n-1}\choose{i-1}}P_{M_{j+i-1}}(\bar{r}_{j+i}-1).\nonumber
\end{eqnarray}
Finally, using (\ref{BExt1a0}) we get  epimorphisms
$$ \ext^{n-i}_R(M_j,R)_\mu \ra  \ext^{n-i}_R(M_j,R)_{\mu +1}\ra 0$$
for all $\mu < \bar{r}_{j+i-1}-n +i-1$. Since $\ext^{n-i}_R(M_j,R)_\mu =0$ for $\mu \ll 0$, this yields $\ext^{n-i}_R(M_j,R)_\mu =0$ for all $\mu \leq \bar{r}_{j+i-1}-n + i-1$. Hence $\indeg(\ext^{n-i}_R(M_j,R) \geq \bar{r}_{j+i-1}-n + i$, as required.
\end{proof}

In particular,
$$
\dim_k  \ext^{n-d}_R(M,R)_{\mu}\leq {{\mu +\bar{r}_{d-1}+n-1}\choose{d-1}}\deg M,
$$
$$
\begin{array}{rl}
\dim_k  \ext^{n-d+1}_R(M,R)_{\mu}&\leq {{\mu +\bar{r}_{d-2}+n-1}\choose{d-2}}P_{M_{d-2}}(\bar{r}_{d-1}-1)
\end{array}
$$
and the numbers $\bar{r}_{d-1}\leq r_d$ and $\bar{r}_{d-2}\leq r_{d-1}$ can be quite sharply
estimated from the degrees of generators and relations of $M$ by \cite[2.1]{CFN}.

For later use we also need a bound  in terms of the Hilbert function.

\bco \label{BoundExt2}
Keep the notation of Theorem \ref{BoundExt1}. Then for $i>0$, 
$$ \dim_k  \ext^{n-i}_R(M,R)_{\mu}\leq {{\mu +\bar{r}_{i-1} +n -1}\choose{i-1}}H_{\ol{M_{i-1}}}(\bar{r}_{i-1}) \leq {{\mu +\bar{r} +n -1}\choose{i-1}}H_{\ol{M}}(\bar{r}). $$
\eco

\begin{proof} The second inequality follows from the first one by using  the fact $\bar{r}_{i-1} \leq \bar{r}$.

To prove the first inequality,  first note from (\ref{BExt1b}) that $\dim_k  \ext^{n-1}_R(M_j,R)_{\mu} \leq P_{M_j}(\bar{r}_j)$. Using this inequality  instead of (\ref{BExt1c}) in the last induction step of the above Theorem we get
$$
\dim_k  \ext^{n-i}_R(M,R)_{\mu}\leq {{\mu +\bar{r}_{i-1}+n-1}\choose{i-1}}\Delta^{i-1}P_{M}(\bar{r}_{i-1}).
$$
 Further, note that  $\Delta^{i-1}P_{M}(t) = P_{M_{i-1}}(t)$. Since this polynomial
 is  increasing for all $t\geq \bar{r}_{i-1} $ 
and  $P_{M_{i-1}}(\bar{r}_{i-1}) = H_{\ol{M_{i-1}}}(\bar{r}_{i-1})$ (by Lemma \ref{BoundExt3}(i)), the claim follows from  the above inequality.

\end{proof}

\begin{Lemma}  \label{Hilb1} Assume that $M$ is  a finitely generated graded $R$-module of dimension $d\geq 1$ and  $\indeg M =0$. Let $l_1,\ldots ,l_d$ be a filter regular sequence of linear forms on $M$ and $B= \dim_k(M/(l_1,...,l_d)M)$. Then 
\begin{itemize}
\item[(i)] $H_M(\mu ) \leq B{\mu +d-1 \choose d-1}$,
\item[(ii)] $H_M(\mu ) \leq \mu (M){\mu +n-1\choose n-1}$.
\end{itemize}
\end{Lemma}

\begin{proof} (i). We do induction on $d$. Let $d=1$. From the exact sequence
$$0 \rightarrow  (0:l_1)_{\mu -1} \rightarrow M_{\mu -1} \rightarrow  M_\mu  \rightarrow 
(M/l_1M)_\mu  \rightarrow 0,$$
and $M_{-1} =0$, we get
$$\begin{array}{ll}
\dim_k(M_\mu ) &\leq \dim_k(M_{\mu -1}) + \dim_k((M/l_1M)_\mu )\\
& \leq \cdots \leq \sum_{j=0}^\mu  \dim_k((M/l_1M)_j)  \leq \dim_k(M/l_1M) = B.
\end{array}$$
Let $d\geq 2$. As above
$$\dim_k(M_\mu ) \leq  \sum_{j=0}^\mu  \dim_k((M/l_dM)_j)  .$$
An application of   the induction hypothesis yields
$$\dim_k(M_\mu ) \leq  \dim_k(M/(l_1,...,l_d)M) \sum_{j=0}^s {j +d-2 \choose d-2} = B
{\mu +d-1 \choose d-1} .$$

(ii). This is clear if we present $M$ as a factor module of the free module $\oplus_{j=1}^{\mu (M)} R(-a_j)$, where $a_j \geq 0$ is an integer for all $j$.

\end{proof}

The following bounds do not depend on the Hilbert function of $M$. It is an extension of  \cite[Theorem 3.4]{H} to the case of modules. Note that our proof here is completely different from that in \cite{H}.

\begin{Theorem} \label{Hilb2} Assume that $M$ is  a finitely generated graded $R$-module of dimension $d\geq 1$ and  $\indeg M =0$. Let $l_1,\ldots ,l_d$ be a filter regular sequence of linear forms on $M$, set $M_i :=M/(l_1,...,l_i)M$, $B:= \dim_k(M_d)$ and $\bar{r}_i = \reg(\ol{M_i})$. Then for all $0<i \leq n$ we have
\begin{itemize}
\item[(i)] $\dim_k \Ext^{n-i}_R(M,R)_\mu  \le B {\ol{r}_{i-1}+d-i  \choose d-i}{\mu +\ol{r}_{i-1}+n -1 \choose i-1}$,
\item[(ii)] $\dim_k H^i_{\mfrak}(M)_\mu \le B{\ol{r}_{i-1}+d-i \choose d-i}{-\mu +\ol{r}_{i-1} -1 \choose i-1}$.
\end{itemize}
\end{Theorem}

\begin{proof} The second statement follows from the first one and the isomorphism $ \Hom(H^i_{\mfrak}(M), k) \cong \Ext^{n-i}_R(M,R)(-n)$. To prove the first statement, applying Lemma \ref{Hilb1} to $\ol{M_{i-1}}$ we get
$$H_{\ol{M_{i-1}}} (\ol{r}_{i-1} )\le B {\ol{r}_{i-1}+d-i  \choose d-i}.$$
The result then follows from Corollary \ref{BoundExt2}.
\end{proof}

Write the Hilbert polynomial of $M$ in the  form:
$$P_{M}(t)= e_0(M){t+d-1 \choose d-1} -e_1(M) {t+d-2 \choose d-2}+ \cdots + (-1)^{d-1}e_{d-1}(M).$$
Then $e_0(M), e_1(M),...,e_{d-1}(M)$ are called  {\it Hilbert coefficients}
of $M$. Note that $e_0(M)=\deg(M)$. Applying the above estimates we can bound the Hilbert coefficients in terms of  the Castelnuovo-Mumford regularity of $M$. The following result extends Theorem 4.1 and Theorem 4.6 in \cite{H}. Moreover the bound here is also a little bit better.

\begin{Theorem} \label{Hilb3} Assume that $M$ is  a finitely generated graded $R$-module of dimension $d\geq 1$ and  $\indeg M =0$. Let $l_1,\ldots ,l_d$ be a filter regular sequence of linear forms on $M$ and $B= \dim_k(M/(l_1,...,l_d)M)$. Then for all $0\le i \le d-1$ we have
$$|e_i(M)| \leq B\cdot(\reg(\bar{M}) + 1)^i.$$
\end{Theorem}

\begin{proof} As usual we set $\bar{r} = \reg(\bar{M})$. We do induction on $d$. Note that $0 \le e_0(M) \le B$. Hence the inequality holds true for $i=0$. In particular the statement holds for $d=1$. Assume that the statement holds for all modules of dimension $d-1\geq 1$. Let $M$ be a module of dimension $d$ and $M_1 = M/l_1M$. Then $e_i(M) = e_i(M_1)$  for all $i\le d-2$. Since $\reg(\ol{M_1})\le \reg(\bar{M})$ and $\dim_k(M_1/(l_2,...,l_d)M_1) = B$, by the induction hypothesis it suffices to show the inequality
$$|e_{d-1}(M)| \leq B(\bar{r}+1)^{d-1}.$$
Note that we may assume $M = \bar{M}$, i.e. $H^0_{\mfrak}(M)=0$. From the
Grothendieck-Serre  formula (\ref{EGS})
we get (setting $t =-1$):
$$(-1)^{d-1}e_{d-1}(M) = C_d - D_d,$$
where
$$C_d= \dim_k(H^1_{\mfrak}(M)_{-1} )+ \dim_k(H^3_{\mfrak}(M)_{-1} ) \ \cdots ,$$
and
$$D_d= \dim_k(H^2_{\mfrak}(M)_{-1})+ \dim_k(H^4_{\mfrak}(M)_{-1} ) \ \cdots .$$
By Theorem \ref{Hilb2}(ii) we have
$$C_d \leq B \displaystyle{\sum_{1\leq 2j + 1\leq d} {\bar{r}\choose 2j} {\bar{r} + d -2j - 1 \choose d-2j -1}} =: B.\tilde{C}_d.$$
We show by induction on $d$ that $\tilde{C}_d \le (\bar{r}+1)^{d-1}$.  We have $\tilde{C}_2 = \bar{r} +1$ and $\tilde{C}_3 = \bar{r}^2 + \bar{r} +1 < (\bar{r} +1)^2$. Let $d\ge 4$. Assume that
$$\tilde{C}_{d-1} \leq (\bar{r}+1)^{d-2}.$$
If $d$ is even, then $d-2j-1 \geq 1$ and
$${\bar{r} + d -2j - 1 \choose d-2j -1}  = \frac{\bar{r} + d -2j - 1}{d-2j -1}{\bar{r} + (d-1) -2j - 1 \choose (d-1)-2j -1}\leq (\bar{r}+1){\bar{r} + (d-1) -2j - 1 \choose (d-1)-2j -1}.$$
Hence, by the induction hypothesis on $\tilde{C}_{d-1}$ we get
$$\tilde{C}_d \leq (\bar{r}+1) \displaystyle{\sum_{1\leq 2j + 1\leq d} {\bar{r} \choose 2j} {\bar{r} + (d-1) -2j - 1 \choose (d-1)-2j -1}} = (\bar{r} +1) \tilde{C}_{d-1} \leq (\bar{r} +1)^{d-1}.$$
If $d$ is odd, say $d= 2\delta +1$, then for $j< \delta $ we have $d-2j -1 \geq 2$ and
$${\bar{r} + d -2j - 1 \choose d-2j -1}  =( \frac{\bar{r} }{d-2j -1} +1){\bar{r} + (d-1) -2j - 1 \choose (d-1)-2j -1}\leq (\frac{\bar{r}}{2}+1)  {\bar{r} + (d-1) -2j - 1 \choose (d-1)-2j -1}.$$
Therefore
$$\begin{array}{ll}
\tilde{C}_d &\le (\frac{\bar{r}}{2}+1)  \displaystyle{\sum_{1\leq 2j + 1\leq d-1} {\bar{r}\choose 2j} {\bar{r} + (d-1) -2j - 1 \choose (d-1)-2j -1}} + {\bar{r} \choose d-1}\\
& < \displaystyle{(\frac{\bar{r}}{2}+1) \tilde{C}_{d-1} + \frac{(\bar{r}+1)^{d-2}\bar{r}}{2}} \\
& \leq \displaystyle{(\frac{\bar{r}}{2} +1)(\bar{r}+1)^{d-2}} + (\bar{r}+1)^{d-2}\frac{\bar{r}}{2}\\
 & = (\bar{r}+1)^{d-1}.
\end{array}$$
Thus we have proved $\tilde{C}_d \le  (\bar{r}+1)^{d-1}$, and so  $C_d \le B (\bar{r}+1)^{d-1}$. Similarly, $D_d \le B (\bar{r}+1)^{d-1}$. Hence
$$|e_{d-1}(M) | \le \max\{C_d,\ D_d\}\le B (\bar{r}+1)^{d-1},$$ 
as required.
\end{proof}

\begin{Remark}\label{Hilb4} (i) If $M$ is a Cohen-Macaulay module, then $B = \deg (M)$. In this case the bound of Theorem \ref{Hilb3} is related to the bound given in \cite[Lemma 11]{HHy1}. If $M= R/I$, where $I$ is a homogeneous ideal generated by forms of degrees at most $\Delta $, then  
$$B\le \max\{\Delta^{n-d},\  \adeg(M)^{n-d}\},$$ 
where $\adeg(M)$ is the so-called the arithmetic degree of $M$, see the proof of \cite[Theorem 3.4]{H}. 

(ii) Considering $M/(l_1,...,l_d)M$ as a module over $R/(l_1,...,l_d)R$, by Lemma \ref{Hilb1}(ii), we have 
$$B \leq \mu(M) {\bar{r} + n-d \choose n-d}.$$

(iii) Example 4.9 in \cite{H} shows that the bound on Hilbert coefficients given in the above theorem is rather good.
\end{Remark}

\section{A bound for the homological degree}\smallskip
\label{C}

The homological degree of a finite graded $R$-module $M$ was introduced by
Vasconcelos. It is defined recursively on the dimension as follows:

\begin{Definition}\label{hdeg} {\rm  \cite[Definition 9.4.1]{Va}
The homological degree of $M$ is the number 
\begin{equation}\label{E-hdeg}
\hdeg (M)  = \deg (M) + \sum_{i=0}^{d-1}{d-1 \choose i} \hdeg(\Ext_R^{n+i+1-d}(M,R)).
\end{equation}}
\end{Definition}

Note that

(a) $\hdeg (M) \geq  \deg(M) $, and the equality holds if and only if $M$ is a
Cohen-Macaulay module.

(b) $\hdeg (M) = \hdeg(M/H^0_{\mfrak}(M)) + \dim_k (H^0_{\mfrak}(M))$.

Let $\gen (M)$ denote the maximal degree of elements in a minimal set of
homogeneous generators of $M$. 
It turns out that the homological degree gives an upper bound for the
Castelnuovo-Mumford regularity
$$\reg (M )\leq \gen (M) + \hdeg (M) -1.$$
This result was first proved for rings by Doering,  Gunston and Vasconcelos
(\cite[Theorem 2.4]{DGV}). Later on it was extended to modules by Nagel (\cite[Theorem
3.1]{Na}). It was also shown in \cite{HHy} that one can use $\hdeg(M)$ to bound the 
Castelnuovo-Mumford regularity of Ext modules. In Chapter 9 of  the book \cite{Va} one can 
find some interesting applications of this invariant. Therefore Vasconcelos asked the following 
question (see the last two lines on page 261 of \cite{Va}): 

Is the homological degree bounded by a polynomial  function of the Castelnuovo-Mumford 
regularity?

The following result gives a positive answer to this question.

\begin{Theorem} \label{B-hdeg1} Let $M$ be a non-zero finitely generated graded 
$R$-module of dimension $d>0$. Then
$$\hdeg(M)\leq \left[\mu(M) {\reg (M) - \bg(M) + n\choose n}\right]^{ 2^{(d-1)^2}} .$$
\end{Theorem}

 In order to prove this theorem we need some auxiliary results.

\begin{Lemma} \label{B-deg} (i)
$\deg(M) +\dim_kH^0_{\mfrak}(M) \leq \sum_{\mu = \bg(M)}^{\reg (M)} H_M(\mu ) .$

(ii) $\sum_{\mu = \bg(M)}^{\reg (M)} H_M(\mu ) \leq 
\mu (M) {\reg (M) - \bg(M) + n\choose n}.$
\end{Lemma}

\begin{proof} (i) Let $l_1,...,l_d$ be a generic linear s.o.p.  of $M$ and  $\bar{M}= 
M/H^0_{\mfrak}(M)$.  Note that
$\reg(\bar{M}/ (l_1,...,l_d)\bar{M} ) \leq \reg(\bar{M}) \leq \reg(M)$ and $\bg(\bar{M}/ 
(l_1,...,l_d)\bar{M} ) \geq \bg(\bar{M})\geq \bg(M)$. Hence
$$\begin{array}{ll} \deg(M) & = \deg(\bar{M})  \leq \dim_k ( \bar{M}/ (l_1,...,l_d)\bar{M} ) 
\\
& = \sum_{\mu = \bg(M)}^{\reg (M)}\dim_k[\bar{M}/ (l_1,...,l_d)\bar{M}]_\mu \leq 
\sum_{\mu = \bg(M)}^{\reg (M)} \dim_k(\bar{M}_\mu ).
\end{array}$$
On the other hand, $H^0_{\mfrak}(M)_\mu =0$ for all $\mu <\bg(M)$ and $\mu >\reg(M)$. 
This yields
$$ \deg(M) + \dim_kH^0_{\mfrak}(M) \leq \sum_{\mu = \bg(M)}^{\reg (M)} 
[\dim_k(\bar{M}_\mu ) + \dim_kH^0_{\mfrak}(M)_\mu ] = \sum_{\mu = \bg(M)}^{\reg (M)} 
H_M(\mu ).$$

(ii) We may assume that $\bg(M) =0$.  Then the inequality follows from Lemma \ref{Hilb1}(ii).
\end{proof}

\begin{Lemma} \label{C-hdeg}  Let $ \bar{M} = M/H^0_{\mfrak}(M)$ and $\bar{r}=\reg(\bar{M})$.  
\begin{itemize}
\item[(i)] If $d\leq 1$, then $ \hdeg(M) = \dim_k (H^0_{\mfrak}(M))+ H_{\bar{M}}(\bar{r} )$.
\item[(ii)] If $d\geq 2$, then
$\hdeg(\Ext^{n-1}_R(M,R)) \leq (H_{\bar{M}}(\bar{r} ) - \deg(M))H_{\bar{M}}(\bar{r} )$.
\end{itemize}
\end{Lemma}

\begin{proof}
(i) The statement is trivial for  $d=0$.

If $d=1$, then by (\ref{E-hdeg}) 
$\hdeg(M)= \deg(\bar{M}) + \dim_k (H^0_{\mfrak}(M))$. Since $\dim(\bar{M} )=1$, by Lemma \ref{BoundExt3}(i), $\deg(\bar{M}) = P_{\bar{M}}(\bar{r}) = H_{\bar{M}}(\bar{r} )$. Hence $ \hdeg(M) = \dim_k (H^0_{\mfrak}(M))+ H_{\bar{M}}(\bar{r} )$.

(ii) Let  $d\geq 2$. Without loss of generality we may assume that $M = \bar{M}$, i.e. $\dept(M) > 0$ and hence $r := \reg(M) = \bar{r}$.  For simplicity, let $E_1 = \Ext^{n-1}_R(M,R)$. Since $\dim(E_1) \leq 1$ (see \cite[p. 63]{Sc}), by (\ref{E-hdeg}),
$$\hdeg(E_1) = \dim_k (H^0_{\mfrak}(E_1)) + \deg(E_1).$$
Let $M' = M_{\ge r}$ and $E'_1:= \Ext^{n-1}_R(M',R)$. Using the exact sequence (\ref{Eregext}) we get $\dim_k (H^0_{\mfrak}(E_1)) \leq \dim_k (H^0_{\mfrak}(E'_1))$ and $\deg(E_1) = \deg(E'_1)$.
Hence, by Lemma \ref{B-deg}(i) we get
\begin{equation}  \hdeg(E_1)  \leq  \dim_k (H^0_{\mfrak}(E'_1)) + \deg(E'_1) \leq  \sum_{\mu = \indeg(E'_1)}^{\reg (E'_1)} \dim_k ((E'_1)_\mu) . \label{C-hdega}
 \end{equation}
 Since $\dept (M') >0$ and $\reg(M') = r$, by Theorem \ref{newregextMR}
 $$\reg(E'_1) + n-1 \leq \max\{ P_{M'}(r-1) - r, \ -r-1\}.$$
 Let $y$ be a generic linear form. Note that $r\leq \indeg (M'/yM') \leq \reg(M'/yM') \leq  \reg(M') =r$. This implies that $M'/yM'$ is generated in degree $r$ and by Lemma \ref{BoundExt3}(ii),
 $$H_{\ol{M'/yM'}}(r) \geq \deg (M'/yM') = \deg(M).$$
 Therefore, by Lemma \ref{BoundExt3}(i), we get
 \begin{eqnarray}
 P_{M'}(r-1) & =& P_{M'}(r) - P_{M'/yM'}(r) = H_{M'}(r) - H_{\ol{M'/yM'}}(r) \nonumber\\
  &\leq & H_M(r) - \deg(M). \label{C-hdegb}
  \end{eqnarray}
This yields
$$\reg(E'_1) + n-1 \leq \max\{ H_M(r) - \deg(M) - r, \ -r-1\} \leq H_M(r) - r -1.$$
Thus
$$\reg(E'_1) \leq H_M(r) - r -n.$$
By Theorem \ref{regExtMN}(1),
$$\indeg(E'_1) \geq -r -n +1.$$
By Theorem \ref{BoundExt1} and the inequality (\ref{C-hdegb}),
$$\dim_k((E'_1)_\mu ) \leq P_M(r-1) \leq  H_M(r) - \deg(M),$$
for all $\mu$. Hence, by (\ref{C-hdega}) we finally obtain
$$\hdeg(E_1) \leq (\reg(E'_1) - \indeg(E'_1) + 1) ( H_M(r) - \deg(M)) \leq ( H_M(r) - \deg(M))H_M(r),$$
as required.
\end{proof}

The following result gives a bound on the cohomological degree in terms of the Hilbert polynomial. 

\begin{Theorem} \label{B-hdeg2} Let $M$ be a non-zero finitely generated graded  
$R$-module of dimension $d\geq 1$. Let $ \bar{M} = M/H^0_{\mfrak}(M)$ and $\bar{r}=\reg(\bar{M})$. Then
$$\hdeg(M) \leq \displaystyle \dim_k (H^0_{\mfrak}(M))+ (P_{M}(\bar{r} ))^{2^{(d-1)^2}}.$$
\end{Theorem}

\begin{proof} By Lemma \ref{BoundExt3}(i) it is equivalent to prove that
$$\hdeg(M) \leq \displaystyle \dim_k (H^0_{\mfrak}(M))+ (H_{\bar{M}}(\bar{r} ))^{2^{(d-1)^2}}.$$
We do induction on $d$. For the simplicity, we set $E_j := \Ext_R^{n-j}(M,R)$ and $H := H_{\bar{M}}(\bar{r} )$. The case $d= 1$ was proved in Lemma \ref{C-hdeg}. 

Let $d=2$. Then, by  Lemma \ref{C-hdeg} we get
$$\begin{array}{ll} 
\hdeg(M) & = \dim_k (H^0_{\mfrak}(M)) + \deg(M) + \hdeg(E_1) \\
&\leq \dim_k (H^0_{\mfrak}(M)) + \deg(M) + (H-\deg(M))H \leq \dim_k (H^0_{\mfrak}(M))  + H^2.
\end{array}$$

Let $d\geq 3$. If $H=1$, then from the exact sequence
$$\bar{M}_{\bar{r}-1} \rightarrow \bar{M}_{\bar{r}} \rightarrow (\bar{M}/y \bar{M})_{\bar{r}} \rightarrow 0,$$
where $y$ is a generic linear form, and $(\bar{M}/y \bar{M})_{\bar{r}}\neq 0$ (since $\dim (\bar{M}/y \bar{M})>0$ and $\bar{M}/y \bar{M}$ is generated in degrees at most $\bar{r}$), we get that $\bar{M} \cong R/I(-\bar{r})$ for some homogeneous ideal $I$ and $\reg(R/I) =0$. Hence $I$ is generated by linear forms, and $\bar{M}$ is a Cohen-Macaulay module. In this case, by (\ref{E-hdeg}) 
$$\hdeg(M) = \dim_k (H^0_{\mfrak}(M)) + \deg(M) = \dim_k (H^0_{\mfrak}(M)) + 1,$$
and the above required inequality trivially holds.

From now on we assume that $H\geq 2$. Fix an $i$ such that $2\leq i \leq d-1$. In the sequel we want to bound $\hdeg(E_i)$. Hence, for this part, we may assume that $\dept (M) >0$, and so $r : =\reg(M) = \bar{r}$. 

By   Theorem \ref{newregextMR} and Lemma \ref{BoundExt3}(i), 
$$ \reg (E_i)  \leq (C_{d,d-i} H)^{2^{d-2}} - \bar{r} +1 - n+i,$$
 where $C_{d,j} = \max\{ {d-1\choose j}, {d-1\choose j+1}\}$. Note that
$\sum_{2j\leq d-1}{d-1 \choose 2j} = \sum_{2j+1\leq d-1}{d-1 \choose 2j+1} = 2^{d-2}.$
Therefore  $C_{d,j} \leq 2^{d-2} -1  $ for all $j$ and $d\geq 3$.  Since $H\geq 2$, this implies
$$ C_{d,d-i} H \leq (2^{d-2} - 1)H \leq H^{d-1} - 2 .$$
Hence
\begin{equation} \label{EBh3} \reg (E_i)  \leq (H^{d-1} - 2)^{2^{d-2}} - \bar{r} +1 - n+i,
\end{equation} 
Using Corollary \ref{BoundExt2}, we see that the following holds  for all $\mu \leq \reg(E_i)$
\begin{equation} \label{EBh4} H_{E_i}(\mu  ) \leq   {(H^{d-1} - 2)^{2^{d-2}} + i \choose i-1} H \leq {(H^{d-1} - 2)^{2^{d-2}} + d -1  \choose d-2} H. \end{equation}
Using also the inequality ${a+\delta  \choose \delta} < (a+1)^\delta $ for all $a$ and $\delta \geq 1$, from (\ref{EBh4}) we get
\begin{eqnarray} 
 H_{E_i}(\reg(E_i) ) & \leq &  H ( (H^{d-1}-2)^{2^{d-2}} + 2)^{d-2}\nonumber \\
& \leq & H(H^{(d-1)2^{d-2}} - 2\cdot 2^{d-2} +2 )^{d-2} \nonumber \\
& < & H^{(d-1)(d-2)2^{d-2} +1} .\label{EBh5}
\end{eqnarray}
By induction on $d$ it is easy to check that $(d-2)(d-1)2^{d-2} +1  < 2^{2d-3} -2$ for all $d\geq 3$. Hence, the above inequality yields
\begin{equation} \label{EBh6} H_{E_i}(\reg(E_i) ) <  H^{2^{2d-3} -2} .  \end{equation}
On the other hand, by Theorem \ref{regExtMN}(1), $\indeg(E_i) \geq -\bar{r} - n + i$. Using
Lemma \ref{B-deg}(i) together with (\ref{EBh3}) and (\ref{EBh5}) we have
\begin{eqnarray}\dim_k (H^0_{\mfrak}(E_i)) & < & (\reg(E_i) - \indeg(E_i) +1)H^{(d-1)(d-2)2^{d-2} +1} \nonumber\\
 &\leq & H^{(d-1)2^{d-2}}H^{(d-1)(d-2)2^{d-2} +1} \nonumber\\
&< &  H^{2((d-1)(d-2)2^{d-2} +1)} \nonumber\\
& < & H^{2(2^{2d-3} -2)} .\label{EBh7}
\end{eqnarray}
Since $\dim E_i \leq d-1$ (see \cite[p. 63]{Sc}), by the induction hypothesis, (\ref{EBh6}) and (\ref{EBh7}) we get
\begin{eqnarray} \hdeg(E_i) & < & H^{2(2^{2d-3} -2)} + (H^{2^{2d-3} -2} )^{2^{(d-2)^2}}\nonumber \\
& \leq & 2H^{2^{(d-1)^2} -2^{(d-2)^2+1}}  \nonumber \\
& \leq & 2\frac{ H^{2^{(d-1)^2}}}{2^{2^{(d-2)^2+1}}}.\label{EBh8}
\end{eqnarray}

Now we are ready to estimate $\hdeg(M)$. Using (\ref{E-hdeg}), (\ref{EBh7}), (\ref{EBh8}) and  Lemma \ref{C-hdeg}, we finally get
$$\begin{array}{ll}
\hdeg(M) &= \deg(M) + \dim_k (H^0_{\mfrak}(M)) + \hdeg(E_1) + \sum_{i=2}^{d-1} {d-1\choose i}\hdeg(E_i)\\
& < \dim_k (H^0_{\mfrak}(M)) + \deg(M) + (H-\deg(M))H + \displaystyle \sum_{i=2}^{d-1} {d-1\choose i}2\frac{ H^{2^{(d-1)^2}}}{2^{2^{(d-2)^2+1}}} \\
&< \dim_k (H^0_{\mfrak}(M)) + \displaystyle 2^d\frac{ H^{2^{(d-1)^2}}}{2^{2^{(d-2)^2+1}}}  \\
&< \dim_k (H^0_{\mfrak}(M)) +\displaystyle H^{2^{(d-1)^2}}.
\end{array}$$
In the last estimation we have used the obvious inequality $2^{(d-2)^2+1} >d$ for all $d\geq 3$.
\end{proof}

Now we can prove Theorem \ref{B-hdeg1} as follows:
\vskip0.5cm

\begin{proof}[Proof of Theorem \ref{B-hdeg1}] Set $\d :=\reg (M)-\indeg (M)$.  If $\d =0$, then $r=\overline r$,  $\mu (M)=H_M (r)=H_{\overline M}(\overline r )=P_M (\overline r)+ \dim_k (H^0_{\mfrak}(M))$ and the result follows from Theorem \ref{B-hdeg2}. If $\d >0$, by Theorem \ref{B-hdeg2}, Lemma \ref{B-deg} and Lemma \ref{Hilb1} we have
$$\begin{array}{ll}
\hdeg(M) &\leq \mu (M){\d+ n \choose n} +  \left[\mu (M){\d+ n-1 \choose n-1}\right]^{ 2^{(d-1)^2}} \\
& \leq \mu (M){\d + n \choose n} +  \left[\mu (M){\d + n \choose n} - 1\right]^{ 2^{(d-1)^2}} \\
& \leq \left[\mu (M){\d + n \choose n} \right]^{ 2^{(d-1)^2}},
\end{array}$$
as required.
\end{proof}

As an immediate consequence of Theorem \ref{B-hdeg1} we obtain
\begin{Corollary} \label{B-hdeg3} Let $I$ be a homogeneous ideal of $R$. Then
$$\hdeg(R/I) \leq \left[{\reg (R/I)  + n\choose n}\right]^{ 2^{(d-1)^2}} .$$
\end{Corollary}

\end{document}